\documentclass{amsart}
\usepackage{amsmath,latexsym,amssymb,amsfonts,hyperref}

\newtheorem{theorem}{Theorem}[section]

\newtheorem{lemma}[theorem]{Lemma}

\theoremstyle{definition}
\newtheorem{remark}[theorem]{Remark}

\begin{document}
\title[]{Invariant projections for operators that are free over the diagonal}
\author[]{Serban Teodor Belinschi}
\begin{abstract}
Motivated by recent work of Au, C\'ebron, Dahlqvist, Gabriel, and Male, we study regularity properties
of the distribution of a sum of two selfadjoint random variables in a tracial noncommutative probability 
space which are free over a commutative algebra. We give a characterization of the invariant projections of 
such a sum in terms of the associated subordination functions.
\end{abstract}
\maketitle

\section{Introduction}

Voiculescu's analytic theory of operator-valued free probability \cite{V1995,V2000,FreeMarkov} proved 
numerous times its essential role in the study of operator-valued distributions and freeness
with amalgamation, and in their applications to random matrix theory (see, for instance, \cite{Sh,
HT,M,BMS,HRS}). Recently, a new application of freeness with amalgamation to random matrix
theory has been found by Au, C\'ebron, Dahlqvist, Gabriel, and Male: they show in \cite{ACDGM}
that independent permutation-invariant matrices are asymptotically free with amalgamation over 
the diagonal \cite[Theorems 1.2, 2.2]{ACDGM}. Motivated mainly by this result, we investigate
in this short note the free additive convolution of operator-valued distributions with 
values in a commutative von Neumann algebra.

More specifically, we consider a tracial von Neumann algebra $(\mathcal A,\tau)$ containing
an Abelian von Neumann subalgebra $\mathcal L$, and the unique trace-preserving
conditional expectation $E\colon\mathcal A\to\mathcal L$. We assume that $X=X^*,
Y=Y^*\in\mathcal A$ are free with amalgamation over $\mathcal L$. We assume that
$X+Y$ has a nonzero invariant projection: there exists $a\in\mathbb R$ and $p=p^*=p^2\in
\mathcal A\setminus\{0,1\}$ such that $(X+Y)p=p(X+Y)=ap$. We ask whether this hypothesis
imposes the existence of an invariant projection of $X$ and/or $Y$. This question was first answered 
in the case of scalar-valued distributions (i.e. when $\mathcal L=\mathbb C\cdot1$) by 
Bercovici and Voiculescu in \cite{BVReg}: the existence of $p$ requires the existence of
an invariant projection $q$ for $X$ ($Xq=qX=a_1q$) and $r$ for $Y$ ($Yr=rY=a_2r$) such that 
$\tau(p)+1=\tau(q)+\tau(r)$ and $a=a_1+a_2$ (see \cite[Theorem 7.4]{BVReg}). The proof 
uses the analytic subordination functions of Voiculescu and Biane \cite{V3,B}.

In this note, we provide a characterization in terms of boundary properties of Voiculescu's 
operator-valued subordination functions \cite{V2000,FreeMarkov} of elements $X,Y$ for 
which the above hypothesis is satisfied (see Theorem \ref{main} below). 
Our result is nowhere near as satisfying as \cite[Theorem 7.4]{BVReg}, but one could not
reasonably expect it to be: the reader is invited to consider the case when $\mathcal L$ is
isomorphic to the von Neumann algebra $L^\infty([0,1])$ and recall that any two real-valued
elements in $\mathcal L$ are tautologically free, in order to construct a simple example of
elements $X,Y\in\mathcal L$ which are not constant on any Borel set of positive measure, but
whose sum is constant on any desired Borel set of positive measure. 

In recent years there were numerous results on the {\em lack} of invariant projections 
\cite{SS,CS,MSW,BM}, as well as the occurence of ``trivial'' (in the above sense) invariant
projections \cite{SS,MSY}. As of now, we are not aware of results that indicate the existence
and properties of invariant projections for $X,Y$.

\section{Analytic tools}

Consider a tracial von Neumann algebra $(\mathcal A,\tau)$, and assume that $\mathcal L$
is an Abelian von Neumann subalgebra of $\mathcal A$. We shall assume throughout the paper
that $\mathcal A$ acts on the Hilbert space $\mathcal H:=L^2(\mathcal A,\tau)$, which is the completion 
of $\mathcal A$ with respect to the inner product $\langle \xi,\eta\rangle=\tau(\eta^*\xi)$.
It is known (see, for instance, \cite{Takesaki1}) that there exists a unique trace-preserving 
conditional expectation $E\colon\mathcal A\to\mathcal L$ which is the restriction to
$\mathcal A$ of the orthogonal projection from $L^2(\mathcal A,\tau)$ onto $L^2(\mathcal L,\tau)$.
If $T\in\mathcal A$, we write $T\ge0$ if $T=T^*$ and the spectrum $\sigma(T)\subseteq[0,+\infty)$,
and we write $T>0$ to signify that $T\ge0$ and $\sigma(T)\subseteq(0,+\infty)$. For any
$T\in\mathcal A$, there exists a decomposition in real and imaginary parts: $T=\Re T+i\Im T$, where
$\Re T=\frac{T+T^*}{2}$ and $\Im T=\frac{T-T^*}{2i}$. We define 
$H^+(\mathcal A)=\{T\in\mathcal A\colon\Im T>0\}$, and similar for $\mathcal L$ and any
other von Neumann subalgebra of $\mathcal A$.

Assume that $X=X^*,Y=Y^*\in\mathcal A$ are free over $\mathcal L$ with respect to $E$
(see \cite{V1995}). Define the analytic map
$$
G_X\colon H^+(\mathcal L)\to H^-(\mathcal L),\quad
G_X(b)=E\left[(b-X)^{-1}\right].
$$
As shown in \cite{FAQ2}, $G_X$ is a free noncommutative map in the sense of \cite{KVV},
whose matricial extension fully encodes the distribution of $X$ with respect to $E$. It is also
known that $G_X$ extends to a ``neighborhood of infinity:'' if $\|b^{-1}\|<\|X\|^{-1}$, then 
$G_X(b)=\sum_{n=0}^\infty E\left[b^{-1}(Xb^{-1})^n\right]$ converges in norm, so $w\mapsto
G_X(w^{-1})$ extends as an analytic map to the ball of center zero and radius $1/\|X\|$.

Let $\mathcal L\langle X\rangle$ denote the von Neumann algebra generated by $\mathcal L$ and $X$.
Denote by $E_X\colon\mathcal A\to\mathcal L\langle X\rangle$ the unique trace-preserving conditional 
expectation from $\mathcal A$ to $\mathcal L\langle X\rangle$. It is shown in \cite{V2000} that
there exists a free noncommutative analytic map $\omega_1\colon H^+(\mathcal L)\to H^+(\mathcal L)$
such that
\begin{equation}\label{Vsubord}
E_X\left[(b-X-Y)^{-1}\right]=(\omega_1(b)-X)^{-1},\quad b\in H^+(\mathcal L)\text{ or }
\|b^{-1}\|<\|X+Y\|^{-1}.
\end{equation}
A similar statement holds for a map $\omega_2$, with $X$ and $Y$ interchanged. By applying $E$ 
to \eqref{Vsubord} and using Voiculescu's $R$-transform \cite{V1995,FAQ2}, it is shown in
\cite{BMS} that 
\begin{equation}\label{Bsubord}
G_{X+Y}(b)^{-1}=G_X(\omega_1(b))^{-1}=G_Y(\omega_2(b))^{-1}=\omega_1(b)+\omega_2(b)-b,
\quad b\in H^+(\mathcal L).
\end{equation}
(See \cite{BVReg} for the scalar version of this relation.) Obviously, the above relation extends to $b$ such 
that $\|b^{-1}\|<\|X+Y\|^{-1}$. It is also shown in \cite{V2000,BPV} that
\begin{equation}\label{bigr}
\Im\omega_j(b)\ge\Im b,\quad\omega_j(b^*)=\omega_j(b)^*,\quad b\in H^+(\mathcal L),j=1,2.
\end{equation}
Given that $\mathcal L$ is a commutative von Neumann algebra, hence isomorphic to an algebra
of functions, we shall often write in the following $1/b$ or $\frac1b$ instead of $b^{-1}$ for multiplicative
inverses of elements of $\mathcal L$.

As mentioned in the introduction, we shall be concerned with invariant projections for $X+Y$. In 
the following, we characterize these objects in terms of resolvents. Thus, assume $T=T^*\in\mathcal A$.
Denote by $\displaystyle\lim_{\stackrel{z\longrightarrow a}{\sphericalangle}}$ the limit as
$z$ approaches $a\in\mathbb R$ from the complex upper half-plane nontangentially to $\mathbb R$.

\begin{lemma}\label{pro}
If there exists a $p=p^*=p^2\in\mathcal A\setminus \{0\}$ and $a\in\mathbb R$ such that 
$$
\lim_{\stackrel{z\longrightarrow a}{\sphericalangle}}(z-a)(z-T)^{-1}=p
$$
in the strong operator {\rm (so)} topology, then $Tp=pT=ap$. Conversely, if $Tp=pT=ap$, then
$$
{\rm so\text{-}}\lim_{\stackrel{z\longrightarrow a}{\sphericalangle}}(z-a)(z-T)^{-1}=p.
$$
\end{lemma}

\begin{proof}
The essential part of the proof can be found for instance in \cite{BVReg}. We sketch it here for convenience.
For any vector $\xi\in\mathcal H$ of $L^2$-norm equal to one, we write
\begin{eqnarray*}
\left\|(z-a)(z-T)^{-1}\xi\right\|_2^2 & = & \left\langle (z-a)(z-T)^{-1}\xi,(z-a)(z-T)^{-1}\xi
\right\rangle\\
& = & \left\langle\left((x-a)^2+y^2\right)\left((x-T)^2+y^2\right)^{-1}\xi,\xi\right\rangle\\
& = & \int_{\mathbb R}\frac{(x-a)^2+y^2}{(x-t)^2+y^2}\,{\rm d}\mu_{\xi,T}(t),
\end{eqnarray*}
where $z=x+iy$ is the decomposition in real and imaginary parts of $z$ and $\mu_{\xi,T}$
is the distribution of the selfadjoint random variable $T$ with respect to the expectation (state) 
$\cdot\mapsto\langle\cdot\xi,\xi\rangle$. The dominated convergence theorem guarantees that 
$$
\lim_{\stackrel{z\longrightarrow a}{\sphericalangle}}
\int_{\mathbb R}\frac{(x-a)^2+y^2}{(x-t)^2+y^2}\,{\rm d}\mu_{\xi,T}(t)=\mu_{\xi,T}(\{0\}),
$$
allowing us to conclude.
\end{proof}

\begin{remark}\label{E}
The above lemma together with the weak operator continuity of $E,E_X$ allows us to conclude that
$$
\lim_{\stackrel{z\longrightarrow a}{\sphericalangle}}(z-a)E\left[(z-T)^{-1}\right]=E[p],\quad
\lim_{\stackrel{z\longrightarrow a}{\sphericalangle}}(z-a)E_X\left[(z-T)^{-1}\right]=E_X[p],
$$
in the so topology. Similarly, we have 
$$
{\rm so\text{-}}\lim_{\stackrel{z\longrightarrow a}{\sphericalangle}}\Re(z-a)(z-T)^{-1}=p,\quad
{\rm so\text{-}}\lim_{\stackrel{z\longrightarrow a}{\sphericalangle}}\Im(z-a)(z-T)^{-1}=0.
$$
In particular,
$$
{\rm so\text{-}}\lim_{y\searrow0}y(a-T)\left((a-T)^2+y^2\right)^{-1}=0,\quad
{\rm so\text{-}}\lim_{y\searrow0}y^2\left((a-T)^2+y^2\right)^{-1}=p.
$$
\end{remark}

We need one more (very simple) fact about the functions that behave like $\omega$. 

\begin{lemma}\label{conv}
Assume that $f\colon H^+(\mathbb C)\to H^+(\mathcal L)$ is a free noncommutative function
in the sense of \cite{KVV}. For any $a\in\mathbb R$, the {\rm so} limit
$$
\lim_{y\searrow0}\frac{y}{\Im f(a+iy)}
$$
exists and is finite.
\end{lemma}

\begin{proof}
The proof is based on the representation of free noncommutative maps of noncommutative half-planes 
provided by \cite{PV,W}: there exists a completely positive map $\rho\colon\mathbb C\langle\mathcal X
\rangle\to\mathcal L$, an element $A=A^*$ and $B\ge0$ in $\mathcal L$ such that 
$$
f(z)=A+zB+\rho\left[(\mathcal X-z)^{-1}\right],\quad z\in H^+(\mathbb C).
$$
Then $\Im f(z)=\Im zB+\rho\left[(\mathcal X-z)^{-1}\Im z(\mathcal X-\bar{z})^{-1}\right]=
\Im zB+\rho\left[\frac{\Im z}{(\mathcal X-\Re z)^2+(\Im z)^2}\right]$.
Here $\mathcal X$ is a selfadjoint operator. Thus, 
$$
y(\Im f(a+iy))^{-1}=\left(B+\rho\left[(\mathcal X-a-iy)^{-1}(\mathcal X-a+iy)^{-1}\right]\right)^{-1}.
$$
Trivially the map $y\mapsto(\mathcal X-a-iy)^{-1}(\mathcal X-a+iy)^{-1}$ is decreasing. This
concludes the proof.
\end{proof}

Observe that commutativity of $\mathcal L$ plays no role in the proof of the previous lemma.

In this paper we shall make use also of the estimate 
\begin{eqnarray}
\lefteqn{
\left(\Im f(z)\right)^{-\frac12}(f(z)-f(w))\left(\Im f(w)\right)^{-1}(f(z)-f(w))^*\left(\Im f(z)\right)^{-\frac12}
}\nonumber\\
& \leq & \left\|\left(\Im z\right)^{-\frac12}(z-w)\left(\Im w\right)^{-\frac12}\right\|^2\label{estim},
\quad\quad\quad\quad\quad\quad\quad\quad\quad\quad
\end{eqnarray}
proven in \cite[Proposition 3.1]{LMS} for an arbitrary free noncommutative map $f$ between two 
noncommutative upper half-planes of two $C^*$-algebras. Since $\mathcal L$
is commutative, we sometimes write the above as
$$
\frac{(f(z)-f(w))(f(z)-f(w))^*}{\Im f(z)\Im f(w)}\leq\left\|(\Im z)^{-\frac12}(z-w)(\Im w)^{-\frac12}\right\|^2.
$$

\section{Invariant projections}

Let us re-state our hypotheses: $(\mathcal A,\tau)$ is a tracial von Neumann algebra (with normal
faithful $\tau$), $\mathcal L\subset\mathcal A$ is an Abelian von Neumann subalgebra of $\mathcal A$, 
$E\colon\mathcal A\to\mathcal L$ is the unique trace-preserving conditional expectation from 
$\mathcal A$ to $\mathcal L$, and $X=X^*,Y=Y^*\in\mathcal A$ are two bounded selfadjoint 
random variables which are free with respect to $E$ over $\mathcal L$. Also, $\mathcal L\langle X\rangle$ 
(respectively $\mathcal L\langle Y\rangle$) is the von Neumann algebra generated by $\mathcal L$ and 
$X$ (respectively $\mathcal L$ and $Y$), and $E_X\colon\mathcal A\to\mathcal L\langle X\rangle$
(resp. $E_Y\colon\mathcal A\to\mathcal L\langle Y\rangle$) is the unique trace-preserving conditional
expectation from $\mathcal A$ onto $\mathcal L\langle X\rangle$ (resp. $\mathcal L\langle Y\rangle$).
Finally, $\mathcal A$ acts (faithfully) on the Hilbert space $\mathcal H:=L^2(\mathcal A,\tau)$, which is the 
completion of $\mathcal A$ with respect to the inner product $\langle \xi,\eta\rangle=\tau(\eta^*\xi)$.

We assume that there exists $a\in\mathbb R$ and $p=p^*=p^2\in\mathcal A\setminus\{0\}$ such that
$$
(X+Y)p=p(X+Y)=ap.
$$
As seen in Lemma \ref{pro}, we have so-$\displaystyle\lim_{\stackrel{z\longrightarrow a}{\sphericalangle}}
(z-a)(z-X-Y)^{-1}=p$, and, by Remark \ref{E},
$$
\lim_{\stackrel{z\longrightarrow a}{\sphericalangle}}(z-a)E\left[(z-X-Y)^{-1}\right]=E[p],\quad
\lim_{\stackrel{z\longrightarrow a}{\sphericalangle}}(z-a)E_X\left[(z-X-Y)^{-1}\right]=E_X[p].
$$
(Obviously a similar statement holds if we interchange $X$ and $Y$.) Using \eqref{Vsubord} and the above, 
\begin{eqnarray}
E_X[p] & = & \lim_{\stackrel{z\longrightarrow a}{\sphericalangle}}(z-a)E_X\left[(z-X-Y)^{-1}\right]\nonumber\\
& = & \lim_{\stackrel{z\longrightarrow a}{\sphericalangle}}(z-a)(\omega_1(z)-X)^{-1}\nonumber\\
& = & \lim_{\stackrel{z\longrightarrow a}{\sphericalangle}}\frac{(z-a)}{\sqrt{\Im\omega_1(z)}}
\left(i-\frac{1}{\sqrt{\Im\omega_1(z)}}(X-\Re\omega_1(z))\frac{1}{\sqrt{\Im\omega_1(z)}}\right)^{-1}
\frac{1}{\sqrt{\Im\omega_1(z)}}\nonumber\\
& = & \lim_{y\searrow0} \left[\sqrt{\frac{y}{\Im\omega_1(a+iy)}}\right.\nonumber\\
& & \mbox{}\times iy
\left(iy-\sqrt{\frac{y}{\Im\omega_1(a+iy)}}(X-\Re\omega_1(a+iy))\sqrt{\frac{y}{\Im\omega_1(a+iy)}}\right)^{-1}\nonumber\\
& & \mbox{}\times\left.\sqrt{\frac{y}{\Im\omega_1(a+iy)}}\right].\label{huh}
\end{eqnarray}
Applying $E$ to the above yields
\begin{eqnarray}
E[p] & = &\lim_{y\searrow0} \frac{y}{\Im\omega_1(a+iy)}\nonumber\\
& & \mbox{}\times iyE\left[
\left(iy-\sqrt{\frac{y}{\Im\omega_1(a+iy)}}(X-\Re\omega_1(a+iy))\sqrt{\frac{y}{\Im\omega_1(a+iy)}}\right)^{-1}\right].\label{uhu}
\end{eqnarray}
All limits take place in the so topology.

Using again Remark \ref{E} and the fact that 
$$
\Im\left((\omega_1(z)-X)^{-1}\right)=
-\left(\Im\omega_1(z)+(X-\Re\omega_1(z))(\Im\omega_1(z))^{-1}(X-\Re\omega_1(z))\right)^{-1},
$$
we obtain
\begin{eqnarray}
E_X[p] & = & \lim_{y\searrow0}yE_X\left[\frac{y}{(a-X-Y)^2+y^2}\right]=\nonumber
-\lim_{y\searrow0}y\Im E_X\left[(a+iy-X-Y)^{-1}\right] \\
& = & -\lim_{y\searrow0} y\Im(\omega_1(a+iy)-X)^{-1}\nonumber\\
& = & \lim_{y\searrow0} \sqrt\frac{y}{\Im\omega_1(a+iy)}\label{sieben}\\
& & \mbox{}\times 
\left(1+\left(\sqrt{\frac{y}{\Im\omega_1(a+iy)}}(X-\Re\omega_1(a+iy))\sqrt{\frac{y}{\Im\omega_1(a+iy)}}\right)^2\right)^{-1}\nonumber\\
& & \mbox{}\times\sqrt\frac{y}{\Im\omega_1(a+iy)}\nonumber\\ 
& \leq & \lim_{y\searrow0}\frac{y}{\Im\omega_1(a+iy)}.
\end{eqnarray}
Applying $E$ to the above yields
$$
E[p]\leq\lim_{y\searrow0}\frac{y}{\Im\omega_1(a+iy)}.
$$

\begin{remark}\label{3.1}
Ideally (as it will become clear from our proof below), we would wish that 
$\ker\lim_{y\searrow0}\frac{y}{\Im\omega_1(a+iy)}=\{0\}$. That is obviously implied by 
$\ker E[p]=\{0\}$. Observe that if $0\neq q=q^*=q^2=\ker E[p]$, then $E[qpq]=
qE[p]q=0$, which implies $\tau(qpq)=\tau(E[qpq])=\tau(0)=0$, so that $qpq=0$. Since $p$ is
also a projection, we conclude from the faithfulness of $\tau$ that $pq=qp=0$, so that $p\perp q$,
or, equivalently, $p\leq q^\perp$. This means that there exists a {\em nontrivial algebraic relation}
between an element from $\mathcal L\setminus\mathbb C\cdot1$, namely $q$, and an element 
from $\mathbb C\langle X+Y\rangle\setminus\mathbb C\cdot1$, namely $p$: $pq=qp=0$.

Conversely, let us assume that $o_1\!=\!\ker\lim_{y\searrow0}\frac{y}{\Im\omega_1(a+iy)}\neq\!\{0\}$.
Then $\ker E_X[p]\ge\ker\lim_{y\searrow0}\frac{y}{\Im\omega_1(a+iy)}$, so that there exists an 
element $o_1=o_1^*=o_1^2\in\mathcal L\setminus\mathbb C\cdot1$ such that $o_1$ and the element
$E_X[p]\in\mathbb C\langle X\rangle\setminus\mathbb C\cdot1$ satisfy a
 {\em nontrivial algebraic relation}: $o_1E_X[p]=E_X[p]o_1=0$.

\end{remark}

We study next the nontangential limit of the real part of $\omega_1$ (and thus also of $\omega_2)$
at $a$. A few steps in this proof will not depend on the commutativity of $\mathcal L$. Fix 
$c\in\mathbb R,c\ge2\|X+Y\|$ and $y'\in(0,+\infty).$ We use inequality \eqref{estim}, applied
to $f=\omega_1,z=c+iy',w=a+iy$ in order to write
\begin{eqnarray*}
\lefteqn{
\left[\frac{1}{\Im\omega_1(c+iy')}\right]^\frac12(\omega_1(c+iy')-\omega_1(a+iy))
\left[\frac{1}{\Im\omega_1(a+iy)}\right]}\\
& & \mbox{}\times(\omega_1(c+iy')-\omega_1(a+iy))^*
\left[\frac{1}{\Im\omega_1(c+iy')}\right]^\frac12\\
& \leq & \left\|\frac{1}{\sqrt{y'}}(c-a+iy-iy')\frac{1}{\sqrt{y}}\right\|^2.\quad\quad\quad\quad\quad\quad
\quad\quad\quad\quad\quad\quad
\end{eqnarray*}
This implies
\begin{eqnarray*}
\lefteqn{
(\omega_1(c+iy')-\omega_1(a+iy))
\left[\frac{y}{\Im\omega_1(a+iy)}\right](\omega_1(c+iy')-\omega_1(a+iy))^*\quad}\\
& \leq & \left\|c-a+iy-iy'\right\|^2\frac{\Im\omega_1(c+iy')}{y'}.\quad\quad\quad\quad\quad\quad\quad
\quad\quad\quad\quad\quad
\end{eqnarray*}
We know that $\omega_1$ is analytic around $c$ and takes selfadjoint values, so we may let $y'\to0$ to 
obtain
$$
(\omega_1(c)-\omega_1(a+iy))
\left[\frac{y}{\Im\omega_1(a+iy)}\right](\omega_1(c)-\omega_1(a+iy))^*
\leq\|c-a+iy\|^2\omega_1'(c)(1).
$$
Expanding in real and imaginary parts, we obtain
\begin{eqnarray*}
\lefteqn{
(\omega_1(c)-\Re\omega_1(a+iy))
\left[\frac{y}{\Im\omega_1(a+iy)}\right]
(\omega_1(c)-\Re\omega_1(a+iy))}\\
& & \mbox{}+y\Im\omega_1(a+iy)\ \leq\ \|c-a+iy\|^2\omega_1'(c)(1).\quad\quad\quad\quad
\end{eqnarray*}
We conclude that 
$$
\left\|(\omega_1(c)-\Re\omega_1(a+iy))
\left[\frac{y}{\Im\omega_1(a+iy)}\right]^\frac12\right\|\leq
\|c-a+iy\|\sqrt{\|\omega_1'(c)\|},
$$
so that, by elementary properties of the norm, and recalling that $\frac{y}{\Im\omega_1(a+iy)}\leq1$,
\begin{equation}\label{daleth}
\left\|\Re\omega_1(a+iy)
\left[\frac{y}{\Im\omega_1(a+iy)}\right]^\frac12\right\|\leq
\|c-a+iy\|\sqrt{\|\omega_1'(c)\|}+\|\omega_1(c)\|,
\end{equation}
independently of $y>0$. The bound from the above relation, while necessary, is not sufficient
for our purposes. We need to show that $\lim_{y\to0}\left[\frac{y}{\Im\omega_1(a+iy)}\right]^\frac12
\Re\omega_1(a+iy)\left[\frac{y}{\Im\omega_1(a+iy)}\right]^\frac12$ exists in the so topology and is 
finite. Clearly, this is implied by the existence of 
\begin{equation}\label{soo}
\textrm{so -}\lim_{y\to0}\left[\frac{y}{\Im\omega_1(a+iy)}\right]^\frac12
\omega_1(a+iy)\left[\frac{y}{\Im\omega_1(a+iy)}\right]^\frac12.
\end{equation}
We write:
\begin{eqnarray*}
\lefteqn{\left[\frac{y'}{\Im\omega_1(a+iy')}\right]^\frac12
\omega_1(a+iy')\left[\frac{y'}{\Im\omega_1(a+iy')}\right]^\frac12}\\
& & \mbox{}-
\left[\frac{y}{\Im\omega_1(a+iy)}\right]^\frac12
\omega_1(a+iy)\left[\frac{y}{\Im\omega_1(a+iy)}\right]^\frac12\\
& = & 
\left[\frac{y'}{\Im\omega_1(a+iy')}\right]^\frac12
\omega_1(a+iy')\left(\left[\frac{y'}{\Im\omega_1(a+iy')}\right]^\frac12-
\left[\frac{y}{\Im\omega_1(a+iy)}\right]^\frac12\right)\\
& & \mbox{}+\left[\frac{y'}{\Im\omega_1(a+iy')}\right]^\frac12
(\omega_1(a+iy')-\omega_1(a+iy))\left[\frac{y}{\Im\omega_1(a+iy)}\right]^\frac12\\
& & \mbox{}+\left(\left[\frac{y'}{\Im\omega_1(a+iy')}\right]^\frac12-
\left[\frac{y}{\Im\omega_1(a+iy)}\right]^\frac12\right)
\omega_1(a+iy)\left[\frac{y}{\Im\omega_1(a+iy)}\right]^\frac12.
\end{eqnarray*}
Recalling that
$$
\left\|\left[\frac{y'}{\Im\omega_1(a+iy')}\right]^\frac12
(\omega_1(a+iy')-\omega_1(a+iy))\left[\frac{y}{\Im\omega_1(a+iy)}\right]^\frac12
\right\|\leq
|y-y'|
$$
assures us that the middle term on the right hand side of the equality above converges in norm
to zero as $y,y'\to0$. As shown in Lemma \ref{conv} above, so-$\lim_{y\to0}\left[
\frac{y}{\Im\omega_1(a+iy)}\right]^\frac12$ exists and is strictly between 0 and 1. Thus,
$$
\textrm{so -}\lim_{y,y'\to0}\left(\left[\frac{y'}{\Im\omega_1(a+iy')}\right]^\frac12-
\left[\frac{y}{\Im\omega_1(a+iy)}\right]^\frac12\right)=0.
$$
Clearly $\omega_1(a+iy)\left[\frac{y}{\Im\omega_1(a+iy)}\right]^\frac12=
\left[y\Im\omega_1(a+iy)\right]^\frac12+
\Re\omega_1(a+iy)\left[\frac{y}{\Im\omega_1(a+iy)}\right]^\frac12$.
Since $\Re\omega_1(a+iy)\left[\frac{y}{\Im\omega_1(a+iy)}\right]^\frac12$
has been shown to be bounded in \eqref{daleth}, and $y\Im\omega_1(a+iy)$ is known to be
uniformly bounded as $y\to0$ by a universal constant depending only on the first
two moments of $X$ and $Y$, it follows that $\omega_1(a+iy)\left[\frac{y}{\Im\omega_1(a+iy)}
\right]^\frac12$ is uniformly bounded as $y\to0$. Generally, if $a_\iota=a_\iota^*\to 0$ in the so 
topology and $b_\iota$ is uniformly bounded in norm, then $b_\iota a_\iota$ converges
to zero in the so topology. Indeed, for any $\xi\in\mathcal H$, 
$\|b_\iota a_\iota\xi\|_2\le\|b_\iota\| \|a_\iota\xi\|_2\le(\sup_\iota\|b_\iota\|)\|a_\iota\xi\|_2\to0$.
This guarantees that the first term on the right hand side of the equality above converges in the so
topology to zero as $y,y'\to0$. Finally, under the above assumptions, $a_\iota b_\iota\to0$ in the
wo topology: $|\langle a_\iota b_\iota\xi,\eta\rangle|=|\langle b_\iota\xi,a_\iota \eta\rangle|
\leq\|b_\iota\xi\|_2\|a_\iota \eta\|_2\le(\sup_\iota\|b_\iota\|)\|a_\iota\eta\|_2\to0$.
This guarantees that the last term on the right hand side of the equality above converges in the wo
topology to zero as $y,y'\to0$. This shows that the family 
$\left\{\left[\frac{y}{\Im\omega_1(a+iy)}\right]^\frac12
\omega_1(a+iy)\left[\frac{y}{\Im\omega_1(a+iy)}\right]^\frac12\right\}_{y>0}$
is Cauchy, hence convergent in the wo topology. Up to this point, we did not need
the fact that $\mathcal L$ is an Abelian von Neumann algebra.
However, since $\omega_1$ takes values in a commutative algebra, it follows trivially that the 
third (last) term on the right hand side of the above relation converges also in the so topology
to zero, which proves the existence and finiteness of the so limit \eqref{soo}.

Let us denote
$$
\varpi_1^\Re(a):=\lim_{y\searrow0}
\left[\frac{y}{\Im\omega_1(a+iy)}\right]^\frac12
\Re\omega_1(a+iy)\left[\frac{y}{\Im\omega_1(a+iy)}\right]^\frac12=
\lim_{y\searrow0}\frac{y\Re\omega_1(a+iy)}{\Im\omega_1(a+iy)},
$$
and
$$
\varpi_1^\Im(a):=\lim_{y\searrow0}\frac{y}{\Im\omega_1(a+iy)},
$$
where the limits are in the so topology. We need one more lemma in order to be able to state and
complete the proof of our main result.

\begin{lemma}\label{cher}
Consider a family $\{Y_n\}_{n\in\mathbb N}\subset\mathcal A$ of selfadjoint elements uniformly bounded
in norm. Assume that {\rm so-}$\lim_{n\to\infty}Y_n=Y$ and that there exists a sequence $\{y_n\}_{n
\in\mathbb N}\subset(0,1)$ converging to zero and an element 
$r\in\mathcal A\setminus\{0\}$ such that
$$
{\rm wo\text{-}}\lim_{n\to\infty}iy_n(iy_n-Y_n)^{-1}=r.
$$
Then $\ker Y\neq0$. Moreover, $Yr=0=rY$.
\end{lemma}

\begin{proof}
We claim that $Yr=0$. Indeed, 
\begin{eqnarray*}
Yr & = & Y\lim_{n\to\infty}iy_n(iy_n-Y_n)^{-1}=\lim_{n\to\infty}{iy_nY}(iy_n-Y_n)^{-1}\\
& = & \lim_{n\to\infty}{iy_nY_n}(iy_n-Y_n)^{-1}+{iy_n(Y-Y_n)}{(iy_n-Y_n)^{-1}}.
\end{eqnarray*}
Since $Y_n=Y_n^*$ and there is an $I>0$ such that $\|Y_n\|<I$ for all $n$, by 
continuous functional calculus the first term is bounded in norm by 
$$
\max_{t\in[-I,I]}\left|\frac{iy_nt}{iy_n-t}\right|=\max_{t\in[-I,I]}\left|\frac{y_nt}{\sqrt{y_n^2+t^2}}\right|
=\frac{y_nI}{\sqrt{y_n^2+I^2}}\to0\text{ as }y_n\searrow0.
$$
If $\xi,\eta\in\mathcal H$, then 
\begin{eqnarray*}
\left|\left\langle{iy_n(Y-Y_n)}{(iy_n-Y_n)^{-1}}\xi,\eta\right\rangle\right|
& = &\left|\left\langle{iy_n}{(iy_n-Y_n)^{-1}}\xi,(Y-Y_n)\eta\right\rangle\right|\\
& \leq & \left\|\frac{iy_n}{iy_n-Y_n}\xi\right\|_2\left\|(Y-Y_n)\eta\right\|_2\to0
\end{eqnarray*}
as $y_n\searrow0$, according to our hypothesis that $Y_n\to Y$ in the so topology. We conclude that
$\langle Yr\xi,\eta\rangle=0$ for all $\xi,\eta\in\mathcal H$, so that $Yr=0$ in $\mathcal L$, as claimed.
Since $r\neq0$, any element $\xi\neq0$ which is in the range of $r$ must belong to the kernel of $Y$.

Showing that $rY=0$ is similar. We have:
\begin{eqnarray*}
rY & = & \lim_{n\to\infty}iy_n(iy_n-Y_n)^{-1}Y\\
& = & \lim_{n\to\infty}{iy_nY_n}(iy_n-Y_n)^{-1}+{iy_n}{(iy_n-Y_n)^{-1}(Y-Y_n)}.
\end{eqnarray*}
The first term tends to zero in norm, while the second term, when applied to $\langle\cdot\xi,\eta\rangle$,
is dominated by $\left\|(Y-Y_n)\xi\right\|_2\left\|iy_n(iy_n+Y_n)^{-1}\eta\right\|_2$, which tends to zero.
\end{proof}

Let us state now our main result.

\begin{theorem}\label{main}
Let $X,Y$ be selfadjoint, free over the commutative von Neumann algebra $\mathcal L$. Assume that there
exists a nonzero projection $p$ and $a\in\mathbb R$ such that $(X+Y)p=p(X+Y)=ap$. Denote by
$\omega_1,\omega_2$ Voiculescu's analytic subordination functions associated to $X$ and $Y$,
respectively. Then:
\begin{enumerate}
\item $\ker\left(\sqrt{\varpi_1^\Im(a)}X\sqrt{\varpi_1^\Im(a)}-\varpi_1^\Re(a)\right)\ominus
\ker\varpi_1^\Im(a)\neq\{0\}$;
\item $\ker\left(\sqrt{\varpi_2^\Im(a)}Y\sqrt{\varpi_2^\Im(a)}-\varpi_2^\Re(a)\right)\ominus
\ker\varpi_2^\Im(a)\neq\{0\}$;
\item
\begin{eqnarray*}\lefteqn{
\!E\left[\ker(\sqrt{\varpi_1^\Im(a)}X\sqrt{\varpi_1^\Im(a)}-\varpi_1^\Re(a))\right]
+E\left[\ker(\sqrt{\varpi_2^\Im(a)}Y\sqrt{\varpi_2^\Im(a)}-\varpi_2^\Re(a))\right]}\\
& \geq & E[p]+\Xi,\quad\quad\quad\quad\quad\quad\quad\quad\quad\quad\quad\quad\quad\quad
\quad\quad\quad\quad\quad\quad\quad\quad\quad\quad\quad
\end{eqnarray*}
where $\Xi=\displaystyle\lim_{y\searrow0}
\frac{(\Im G_{X+Y}(a+iy))^2}{(\Re G_{X+Y}(a+iy))^2+(\Im G_{X+Y}(a+iy))^2}$ is an operator
between $\frac{4E[p]}{4E[p]+1}$ and $1$. We have $\Xi=1$ and equality in the above whenever 
$E[p]>0$.
\end{enumerate}
\end{theorem}

\begin{proof}
Let us return to equality \eqref{huh}: we have
\begin{eqnarray}
E_X[p] 
& = & \text{so-}\lim_{y\searrow0} \left[\sqrt{\frac{y}{\Im\omega_1(a+iy)}}\right.\nonumber\\
& & \mbox{}\times iy
\left(iy-\sqrt{\frac{y}{\Im\omega_1(a+iy)}}(X-\Re\omega_1(a+iy))\sqrt{\frac{y}{\Im\omega_1(a+iy)}}\right)^{-1}\nonumber\\
& & \mbox{}\times\left.\sqrt{\frac{y}{\Im\omega_1(a+iy)}}\right].\nonumber
\end{eqnarray}
As shown above, 
$$
\lim_{y\searrow0}
\sqrt{\frac{y}{\Im\omega_1(a+iy)}}(X-\Re\omega_1(a+iy))\sqrt{\frac{y}{\Im\omega_1(a+iy)}}=
\varpi_1^\Im(a)^\frac12X\varpi_1^\Im(a)^\frac12-\varpi_1^\Re(a),
$$
so-convergence to a bounded selfadjoint element. We have also seen that the family
$\sqrt{\frac{y}{\Im\omega_1(a+iy)}}(X-\Re\omega_1(a+iy))\sqrt{\frac{y}{\Im\omega_1(a+iy)}}$
is uniformly bounded in norm as $y\in(0,1)$. 
Since in the above relation \eqref{huh}, $0\leq E_X[p]\neq0$ and 
$\sqrt{\frac{y}{\Im\omega_1(a+iy)}},y>0,$ is bounded from below by the positive nonzero element $E[p]$,
it follows that the middle factor in the right hand side cannot converge to zero. Also, if 
$\ker\varpi_1^\Im(a)\neq0$, then 
$\ker\left(\sqrt{\varpi_1^\Im(a)}X\sqrt{\varpi_1^\Im(a)}-\varpi_1^\Re(a)\right)\ge\ker\varpi_1^\Im(a)$.
Indeed, we may write
\begin{eqnarray*}
\lefteqn{
\sqrt{\frac{y}{\Im\omega_1(a+iy)}}(X-\Re\omega_1(a+iy))\sqrt{\frac{y}{\Im\omega_1(a+iy)}}}\\
& = & \sqrt{\frac{y}{\Im\omega_1(a+iy)}}X\sqrt{\frac{y}{\Im\omega_1(a+iy)}}\\
& & \mbox{}-\left[\frac{y}{\Im\omega_1(a+iy)}\right]^\frac14\\
& & \mbox{}\times\left(\left[\frac{y}{\Im\omega_1(a+iy)}\right]^\frac14
\Re\omega_1(a+iy)\left[\frac{y}{\Im\omega_1(a+iy)}\right]^\frac14\right)
\left[\frac{y}{\Im\omega_1(a+iy)}\right]^\frac14.
\end{eqnarray*}
We recall from \eqref{daleth} that 
$\left\|\Re\omega_1(a+iy)\left[\frac{y}{\Im\omega_1(a+iy)}\right]^\frac12\right\|$ is uniformly
bounded as $y\to0$. Elementary operator theory informs us that the norm of an operator on a 
Hilbert space dominates its spectral radius, with equality for normal elements. Since
$\sigma\!\left(\!\Re\omega_1(a+iy)\left[\frac{y}{\Im\omega_1(a+iy)}\right]^\frac12
\right)\cup\{0\}
=\sigma\!\left(\left[\frac{y}{\Im\omega_1(a+iy)}\right]^\frac14\!
\Re\omega_1(a+iy)\left[\frac{y}{\Im\omega_1(a+iy)}\right]^\frac14\right)\cup\{0\}$, it follows that
the spectral radius, and hence the norm, of the right-hand side, selfadjoint, operator is uniformly
bounded as $y\to0$. Since  the kernel of a positive operator equals the kernel of any of its positive powers, 
we conclude that if $\varpi_1^\Im(a)\xi=0$, then 
$(\sqrt{\varpi_1^\Im(a)}X\sqrt{\varpi_1^\Im(a)}-\varpi_1^\Re(a))\xi=0$.
Since 
$$
\left\|iy\left(iy-\sqrt{\frac{y}{\Im\omega_1(a+iy)}}(X-\Re\omega_1(a+iy))
\sqrt{\frac{y}{\Im\omega_1(a+iy)}}\right)^{-1}\right\|\le1,\quad y>0,
$$
in a von Neumann algebra, there exists a sequence $y_n$ converging to zero so that the above
converges in the weak operator (wo) topology. Choose such a limit point and call it $r$. (Note that,
in this particular case, the adjoint of the above also converges, and necessarily to $r^*$.)
We have established above that 
$\ker(\sqrt{\varpi_1^\Im(a)}X\sqrt{\varpi_1^\Im(a)}-\varpi_1^\Re(a))\ge\ker\varpi_1^\Im(a)$. If
this inequality were an equality, then the inequality 
$\mathrm{ran}(r)\leq\ker(\sqrt{\varpi_1^\Im(a)}X\sqrt{\varpi_1^\Im(a)}-\varpi_1^\Re(a))$
provided by Lemma \ref{cher},
would imply $\mathrm{ran}(r)\leq\ker\varpi_1^\Im(a).$ In particular, we would obtain that
the right-hand side of \eqref{huh} converges to zero\footnote{If the bounded sequence $r_n$ converges
wo to $r$ and the positive sequence $x_n$ converges so to $x$, then $\langle x_nr_n\xi,\eta\rangle=
\langle r_n\xi,x_n\eta\rangle=\langle r_n\xi,(x_n-x)\eta\rangle+\langle r_n\xi,x\eta\rangle$. By
Cauchy-Schwartz, $\langle r_n\xi,(x_n-x)\eta\rangle\to0$, and by hypothesis
$\langle r_n\xi,x\eta\rangle\to\langle r\xi,x\eta\rangle=\langle xr\xi,\eta\rangle$.}, contradicting
the fact that $p$, and hence $E_X[p]$, is non-zero. Thus, necessarily
$$
\ker\left(\sqrt{\varpi_1^\Im(a)}X\sqrt{\varpi_1^\Im(a)}-\varpi_1^\Re(a)\right)\gneqq\ker\varpi_1^\Im(a).
$$
This way, we conclude that 
$$
\ker\left(\sqrt{\varpi_1^\Im(a)}X\sqrt{\varpi_1^\Im(a)}-\varpi_1^\Re(a)\right)\ominus
\ker\varpi_1^\Im(a)\neq\{0\}.
$$
The statement for $\omega_2$ and $Y$ follows the same way.

Let us establish next the relation between the kernels from items (1) and (2). Let us take the imaginary part
in \eqref{Bsubord} (we use commutativity of $\mathcal L$ in an essential way):
$$
\Im\omega_1(a+iy)+\Im\omega_2(a+iy)=y+\frac{-\Im G_{X+Y}(a+iy)}{(\Re G_{X+Y}(a+iy))^2+(\Im G_{X+Y}(a+iy))^2}.
$$
(Recall that $\Im G_{X+Y}(a+iy)<0$.) We multiply with $-\Im G_{X+Y}(a+iy)$ to obtain
\begin{eqnarray*}
\lefteqn{-\Im\omega_1(a+iy)\Im G_{X+Y}(a+iy)-\Im\omega_2(a+iy)\Im G_{X+Y}(a+iy)}\\
& = & -y\Im G_{X+Y}(a+iy)+\frac{(y\Im G_{X+Y}(a+iy))^2}{(y\Re G_{X+Y}(a+iy))^2+(y\Im G_{X+Y}(a+iy))^2}.
\end{eqnarray*}
It is easy to verify that the right hand side converges when $y\searrow0$, at least along a subsequence.
Let us analyze each of the two terms on the left hand side separately:
\begin{eqnarray*}
\lefteqn{-\Im\omega_1(a+iy)\Im G_{X+Y}(a+iy)
 =  -\Im\omega_1(a+iy)\Im G_{X}(\omega_1(a+iy))}\\
& = & \Im\omega_1(a+iy)\times \\
& & \!E\left[(\Im\omega_1(a+iy)+(\Re\omega_1(a+iy)-X)
(\Im\omega_1(a+iy))^{-1}(\Re\omega_1(a+iy)-X))^{-1}\right]\\
& = & \!E\left[\left(1+\left((\Im\omega_1(a+iy))^{-\frac12}(\Re\omega_1(a+iy)-X)
(\Im\omega_1(a+iy))^{-\frac12}\right)^2\right)^{-1}\right]\\
& = & y^2\!E\left[\left(y^2+\left(\sqrt{\frac{y}{\Im\omega_1(a+iy)}}(\Re\omega_1(a+iy)-X)
\sqrt{\frac{y}{\Im\omega_1(a+iy)}}\right)^2\right)^{-1}\right].
\end{eqnarray*}
We recognize under the expectation the square of the selfadjoint shown to so-converge to 
$\sqrt{\varpi_1^\Im(a)}X\sqrt{\varpi_1^\Im(a)}-\varpi_1^\Re(a)$, and  which has been shown to be
uniformly norm-bounded in $y$. Since the square of a bounded family of selfadjoints converges
whenever the family converges\footnote{If $Y_y=Y_y^*\to Y$ so as $y\to0$, then $Y_y^2-Y^2
=(Y-Y_y)Y+Y_y(Y-Y_y)$, and
$\|Y_y(Y_y-Y)\xi\|_2^2=
\langle (Y_y-Y)Y_y^2(Y_y-Y)\xi,\xi\rangle\leq\|Y_y\|^2\|(Y_y-Y)\xi\|_2^2\to0$. The other term
converges trivially to zero.}, we obtain that 
\begin{eqnarray*}\lefteqn{
\text{so-}\lim_{y\searrow0}\left(\sqrt{\frac{y}{\Im\omega_1(a+iy)}}(\Re\omega_1(a+iy)-X)
\sqrt{\frac{y}{\Im\omega_1(a+iy)}}\right)^2}\\
& = & \left(\sqrt{\varpi_1^\Im(a)}X\sqrt{\varpi_1^\Im(a)}-\varpi_1^\Re(a)\right)^2.\quad\quad\quad
\quad\quad\quad\quad\quad\quad
\end{eqnarray*}
The family
$$\mathcal Z_y:=y^2\left(y^2+\left(\sqrt{\frac{y}{\Im\omega_1(a+iy)}}(\Re\omega_1(a+iy)-X)
\sqrt{\frac{y}{\Im\omega_1(a+iy)}}\right)^2\right)^{-1},\ y>0,$$
is uniformly bounded from above by 1 and positive in the von Neumann algebra $\mathcal L\langle X
\rangle$. Pick, as before, a subsequence $y_n\searrow0$ such that $\mathcal Z_{y_n}$ converges 
wo to an element $s_1\ge0$ in $\mathcal L\langle X\rangle$. As in the proof of Lemma \ref{cher},
we have $(\sqrt{\varpi_1^\Im(a)}X\sqrt{\varpi_1^\Im(a)}-\varpi_1^\Re(a))^2s_1
=s_1(\sqrt{\varpi_1^\Im(a)}X\sqrt{\varpi_1^\Im(a)}-\varpi_1^\Re(a))^2=0$. It is quite clear that
$s_1\neq0$. Indeed, that follows from \eqref{sieben} the same way as above. In particular, we have
$\ker  (\sqrt{\varpi_1^\Im(a)}X\sqrt{\varpi_1^\Im(a)}-\varpi_1^\Re(a))^2\ge s_1,$
and so $\ker  (\sqrt{\varpi_1^\Im(a)}X\sqrt{\varpi_1^\Im(a)}-\varpi_1^\Re(a))\ge s_1$.
Since
\begin{eqnarray*}
E[s_1]+E[s_2]
& = &E[p]+\lim_{y\searrow0}\frac{(y\Im G_{X+Y}(a+iy))^2}{(y\Re G_{X+Y}(a+iy))^2+(y\Im G_{X+Y}(a+iy))^2},
\end{eqnarray*}
the inequality in item (3) of our theorem follows from the monotonicity of $E$.
The limit in the right hand side is easily seen to be between $\frac{4E[p]}{4E[p]+1}$ and $1$. Finally,
if $E[p]>0$, then $\Re\omega_1(a+iy)$ converges as $y\to0$ to a selfadjoint $\omega_1(a)$ (see \cite[
Theorem 2.2]{LMS}), and, according to \cite[Relation (4.2)]{LMS}, $(\Re\omega_1(a+iy)-\omega_1(a))/y
\to0$ as $y\to0$. Then \eqref{Bsubord} yields $\omega_1(a)+\omega_2(a)=a$ and thus
$(\omega_1(a+iy)-\omega_1(a))G_{X}(\omega_1(a+iy))+(\omega_2(a+iy)-\omega_2(a))G_{Y}
(\omega_2(a+iy))=iyG_{X+Y}(a+iy)+1$. We write $(\omega_1(a+iy)-\omega_1(a))G_{X}(\omega_1(a+iy))
=iyG_{X+Y}(a+iy)(\Re\omega_1(a+iy)-\omega_1(a))/y+i\Im\omega_1(a+iy)G_{X}(\omega_1(a+iy))$.
The first term tends to zero. We claim that the second converges to $\ker(\omega_1'(a)^{-\frac12}
(X-\Re\omega_1(a))\omega_1'(a)^{-\frac12})$ (just here, we agree to denote $\omega_1'(a)(1)$
by $\omega_1'(a)$). Indeed,
\begin{eqnarray*}
\lefteqn{i\Im\omega_1(a+iy)G_{X}(\omega_1(a+iy))-iyE\left[\left(iy-\omega_1'(a)^{-\frac12}
(X-\Re\omega_1(a))\omega_1'(a)^{-\frac12}\right)^{-1}\right]}\\
& = &
i\Im\omega_1(a+iy)G_{X}(\omega_1(a+iy))-iy\omega_1'(a)E\left[\left(iy\omega_1'(a)-
(X-\Re\omega_1(a))\right)^{-1}\right]\\
& = & \left[\frac{\Im\omega_1(a+iy)}{y}-\omega_1'(a)\right]iyG_{X+Y}(a+iy)\\
& &\mbox{}+\omega_1'(a)iyE\left[\left(\omega_1(a+iy)-X\right)^{-1}\left(iy\omega_1'(a)-
X+\Re\omega_1(a)\right.\right.\\
& & \mbox{}\left.\left.-\Re\omega_1(a+iy)-i\Im\omega_1(a+iy)+X\right)
\left(iy\omega_1'(a)-
(X-\Re\omega_1(a))\right)^{-1}\right].
\end{eqnarray*}
As shown in \cite[Theorem 2.2]{LMS}, $\frac{\Im\omega_1(a+iy)}{y}$ increases to $\omega'_1(a)$
as $y\searrow0$ (convergence in so topology) and $iyG_{X+Y}(a+iy)\to E[p]$, so the first term goes to
zero. Next, $iy\left(\omega_1(a+iy)-X\right)^{-1}\left[\frac{\Im\omega_1(a+iy)}{y}\,-\,\omega_1'(a)\right]
iy\left(iy\omega_1'(a)-(X-\Re\omega_1(a))\right)^{-1}$ has the first and third factors 
bounded, while the middle one converges to zero in the so topology. Finally, precisely the same statement 
holds for the last product, namely
$iy\left(\omega_1(a+iy)-X\right)^{-1}\frac{\Re\omega_1(a)-\Re\omega_1(a+iy)}{iy}
iy\left(iy\omega_1'(a)-(X-\Re\omega_1(a))\right)^{-1}$. Thus, the above tends to zero in the so
topology, guaranteeing that 
$$
\ker(\omega_1'(a)^{-\frac12}
(X-\Re\omega_1(a))\omega_1'(a)^{-\frac12})+\ker(\omega_2'(a)^{-\frac12}
(Y-\Re\omega_2(a))\omega_2'(a)^{-\frac12})\!=\!1+E[p].
$$
\end{proof}
Note that, under the very favourable hypothesis on $E[p]$, discussed in Remark \ref{3.1}, the result above, 
and its proof, closely parallels the corresponding result and proof from \cite{BVReg}. This seems to justify 
the statement that the Julia-Carath\'eodory derivative is an important tool in the understanding of invariant
projections for sums of random variables which are free over a von Neumann algebra.


\begin{thebibliography}{3}

\bibitem{ACDGM} Au, B., C\'ebron, G., Dahlqvist, A., Gabriel, F., and Male, C. ``Large permutation 
invariant random matrices are asymptotically free over the diagonal.'' {\em Preprint} (2018):
arXiv:1805.07045v1 [math.PR].

\bibitem{BM} Banna, M., and Mai, T.  ``H\"{o}lder Continuity of Cumulative Distribution Functions 
for Noncommutative Polynomials under Finite Free Fisher Information.'' {\em Preprint} (2018):
arXiv:1809.11153v1  [math.PR].

\bibitem{LMS}  Belinschi, S.T. ``A noncommutative version of the Julia-Wolff-Carath\'eodory theorem.'' 
\emph{J. London Math. Soc. (2)} 95, (2017): 541--566.

\bibitem{BPV}
Belinschi, S.T., Popa, M., and Vinnikov, V.
``Infinite divisibility and a 
non-commutative {B}oolean-to-free {B}ercovici-{P}ata 
bijection.'' \emph{J. Funct. Anal.} {262}, no. 1 (2012): 94--123.

\bibitem{BMS} Belinschi, S.T., Mai, T., and Speicher, R. ``Analytic 
subordination theory of operator-valued free additive convolution 
and the solution of a general random matrix problem.'' 
{\em J. reine angew. Math.} 732 (2017): 21--53.

\bibitem{BVReg} {Bercovici, H., and Voiculescu, D.} ``Regularity questions for free convolution.''
\emph{Nonselfadjoint operator algebras, operator theory, and related topics,} {37--47}, {Oper. Theory 
Adv. Appl.} {104}, Birkh\"{a}user, Basel, 1998.

\bibitem{B} Biane, Ph. ``Processes with free increments.'' {\em Math. Z.} 227 (1998): 143--174.

\bibitem{Bruce} Blackadar, B. {\em Operator Algebras.
 Theory of $C^*$-Algebras and von Neumann Algebras. }
Encyclopaedia of Mathematical Sciences, Volume 122. Berlin Heidelberg:
Springer-Verlag  2006.

\bibitem{CS} Charlesworth, I., and Shlyakhtenko, D. ``Free entropy dimension and regularity of 
non-commutative polynomials.'' \emph{J. Funct. Anal.} 271, no. 8 (2016): 2274--2292.

\bibitem{HT} Haagerup, U., and Thorbj{\o}rnsen, S. ``A new application of random matrices: 
Ext$({\rm C}^*_{\rm red}(F_2))$ is not a group.'' \emph{Ann. of Math. (2)} 162, no. 2 (2005): 711--775.

\bibitem{HRS} Helton, W., Rashidi-Far, R., and Speicher, R. ``Operator-valued Semicircular Elements: 
Solving A Quadratic Matrix Equation with Positivity Constraints.'' \emph{Int. Math. Res. Not.}, No. 22 (2007).

\bibitem{KVV} Kaliuzhnyi-Verbovetskyi, D. S., and Vinnikov, V., {\em Foundations of free noncommutative 
function theory.}
Mathematical Surveys and Monographs 199. Providence, RI: American Mathematical Society 2014,

\bibitem{MSW} Mai, T., Speicher, R., and Weber, M. ``Absence of algebraic relations and of zero divisors 
under the assumption of full non-microstates free entropy dimension.'' \emph{Adv. Math.} 304
(2017): 1080--1107.

\bibitem{MSY} Mai, T., Speicher, R., and Yin, S. ``The free field:  zero divisors, Atiyah property and 
realizations via unbounded operators.'' {\em Preprint} (2018): arXiv:1805.04150.

\bibitem{M} Male, C. ``The norm of polynomials in large random and
deterministic matrices.'' With an appendix by Dimitri Shlyakhtenko.
\emph{Probab. Theory Related Fields} {154}, no. 3-4 (2012): 477–-532. 

\bibitem{PV} Popa, M., and Vinnikov, V. ``Non-commutative functions and the non-commutative free 
L\'evy-Hin\v{c}in formula.'' \emph{Adv. Math.} 236 (2013): 131--157.

\bibitem{Sh} Shlyakhtenko, D. ``Random Gaussian band matrices and freeness with amalgamation.'' 
\emph{Internat. Math. Res. Notices} no. 20 (1996): 1013--1025.

\bibitem{SS} Shlyakhtenko, D., and Skoufranis, P. ``Freely Independent Random Variables with 
Non-Atomic Distributions.'' {\em Trans. Amer. Math. Soc.} 367 (2015): 6267--6291.

\bibitem{ShlB} Shlyakhtenko, D.
``Free probability of type $B$ and asymptotics of finite rank 
perturbations of random matrices.'' {\em Indiana Univ. Math. J.} 67 (2018): 971--991.

\bibitem{Takesaki1} Takesaki, M. {\em Theory of
Operator Algebras} I--III. New York: Springer-Verlag Inc., 1979.

\bibitem{V3} Voiculescu, D.
``The analogues of entropy and of Fisher's 
information measure in free probability theory. I.'' \emph{Comm. Math. Phys.} 
{155} (1993): 411--440.

\bibitem{V1995}  Voiculescu, D.
``Operations on certain non-commutative operator-valued
  random variables.'' In {\em Recent advances in
operator algebras} (Orl{\'e}ans) {\em Ast\'erisque} 232 (1995): 243--275. 

\bibitem{V2000}  Voiculescu,  D.
``The coalgebra of the free difference quotient and free probability.'' 
{\em Internat. Math. Res. Notices} 2 (2000): 79--106. 

\bibitem{FreeMarkov}  Voiculescu, D. ``Analytic subordination consequences of free Markovianity.''
{\em Indiana Univ. Math. J.} {51} (2002): 1161--1166.

\bibitem{FAQ2} Voiculescu, D.V. ``Free analysis questions II: The Grassmannian completion and the series 
expansions at the origin.'' {\em J. reine angew. Math.} 645 (2010): 155--236.

\bibitem{W} Williams, J.D. ``Analytic function theory for operator-valued free probability'' 
\emph{J. reine angew. Math.} (2015) DOI 10.1515/crelle-2014-0106

\end{thebibliography}
\end{document}